\documentclass[reqno,12pt]{amsart} 
\usepackage{vmargin}
\setpapersize{A4}
\usepackage{amsmath}

\usepackage{amsfonts}   
\usepackage{amssymb}      
\usepackage{eufrak}      
\usepackage{amscd}      
\usepackage{amsthm}      
\usepackage{tikz,amsmath}
\usetikzlibrary{arrows}
\usepackage{tikz}
\usetikzlibrary{matrix,calc}

\usepackage{epsfig}      
\usepackage{amstext}      
\usepackage{graphicx}

\usepackage[all,line,dvips]{xy} 
\CompileMatrices %goer at xy-matricerne koerer hurtigere.

\newcommand{\Ad}{\operatorname{Ad}}

\newcommand{\id}{\operatorname{id}} 
\newcommand{\Aut}{\operatorname{Aut}}

\newcommand{\diag}{\operatorname{diag}}

\newcommand{\co}{\operatorname{co}}
\newcommand{\Span}{\operatorname{Span}}
\newcommand{\Tr}{\operatorname{Tr}}

\newcommand{\Per}{\operatorname{Per}}

 \newcommand{\supp}{\operatorname{supp}}

\newcommand{\Is}{\operatorname{Is}}

\newcommand{\Sp}{\operatorname{Sp}}

\newcommand{\Ro}{\operatorname{RO}}

   \theoremstyle{plain}%default
   \newtheorem{thm}{Theorem}[section]
   \newtheorem{prop}[thm]{Proposition}
   \newtheorem{lemma}[thm]{Lemma}  
   \newtheorem{cor}[thm]{Corollary}
   \theoremstyle{definition}
   
   \newtheorem{defn}[thm]{Definition}
   \newtheorem{example}[thm]{Example}
   \theoremstyle{remark}
   
   \newtheorem{remark}[thm]{Remark}

\newtheorem{assumption}[thm]{Assumption}
\usepackage{mdframed,xcolor}

\definecolor{mybgcolor}{gray}{0.8}
\definecolor{myframecolor}{rgb}{.647,.129,.149}

\mdfdefinestyle{mystyle}{
  usetwoside=false,
  skipabove=0.6em plus 0.8em minus 0.2em,
  skipbelow=0.6em plus 0.8em minus 0.2em,
  innerleftmargin=.25em,
  innerrightmargin=0.25em,
  innertopmargin=0.25em,
  innerbottommargin=0.25em,
  leftmargin=-.75em,
  rightmargin=-0em,
  topline=false,
  rightline=false,
  bottomline=false,
  leftline=false,
  backgroundcolor=mybgcolor,
  splittopskip=0.75em,
  splitbottomskip=0.25em,
  innerleftmargin=0.5em,
  leftline=true,
  linecolor=myframecolor,
  linewidth=0.25em,
}

\newmdenv[style=mystyle]{important}

\usepackage{lipsum}

   \numberwithin{equation}{section}

        \date{\today}
\title[Factor types and ground states]{}

\date{\today}

\email{matkt@math.au.dk}
\address{Department of Mathematics, Aarhus University, Ny Munkegade, 8000 Aarhus C, Denmark}

\begin{document}

\maketitle

\begin{center}
\begin{large} {\bf KMS weights on graph $C^*$-algebras II}
\end{large}
\end{center}

\begin{center}
 {\bf Factor types and ground states}
\end{center}

\bigskip

\begin{center}
Klaus Thomsen
\end{center}

\section{Introduction}

The study of KMS states on graph $C^*$-algebras has a long
story which began well before the notion of a graph $C^*$-algebra
was coined, starting perhaps with the paper \cite{PS} by Powers and
Sakai, or the paper \cite{OP} by Olesen and Pedersen. For finite
graphs we have now a complete description of all KMS states for generalized
gauge actions, thanks to the efforts of many mathematicians. See \cite{
  aHLRS} and \cite{CT} for the most recent and most comprehensive account. It
seems therefore now time to consider gauge actions and generalised
gauge actions on graph $C^*$-algebras of infinite graphs, and the first
investigations in this direction was performed by Carlsen and Larsen
in \cite{CL}, building on the closely related work by Exel and Laca in
\cite{EL}, and by the author in \cite{Th3}, \cite{Th4},
\cite{Th5}. The two approaches are very different and complement each
other nicely.

The present work is a continuation of \cite{Th5} where the gauge
invariant KMS weights
of generalised gauge actions were
investigated, culminating in a complete description of the KMS weights
for the gauge
action on the $C^*$-algebra of a strongly connected graph with at most countably many
exits. The structure turned out to be very rich; in fact, to a agree
comparable with the structure constructed more than 30 years ago by Bratteli, Elliott and
Kishimoto, \cite{BEK}, based on the classification of
AF-algebras. However, the factor types of the extremal KMS weights
identified in \cite{Th5} were not determined, nor was the ground states of
the systems obtained by restricting the gauge actions to corners given
by vertexes in the graph. It is the purpose of the present paper
to provide this information. At the same time we make an effort to
obtain results valid for actions more complicated than the gauge
action. It will be shown by example that by considering actions other than the
gauge action, the variety of factor types becomes much
greater.

It follows from \cite{Th5} that there are three kinds of extremal KMS
weights for the gauge action when the graph is strongly connected and only has countably
many exits: Boundary KMS weights
coming from infinite emitters in the graph, a conservative KMS weight
coming from a positive eigenvector for the adjacency matrix of the
graph with the smallest possible eigenvalue, and dissipative KMS weights coming
from exits in the graph. The three types are not always all present, but when the graph has at
most countably many exits they comprise all the extremal KMS
weights. The factor type of the extremal boundary KMS weights are
invaribly $I_{\infty}$, for all generalized gauge actions, and we focus therefore on the two other kinds of KMS weights since they exhibit a much more
varied behaviour. In Section \ref{conservative} we determine the factor type
of the essentially unique conservative KMS weight for the gauge action
which exists when
the graph is recurrent with finite Gurevich entropy. The method of
proof is taken from \cite{Th3}, but it is only successful because of the
conservative nature of the corresponding measure which was established
in \cite{Th5}. The result is that the factor type of this weight is
determined by the global period $d_G$ and the Gurevich entropy $h(G)$
of the graph; more precisely its factor 
type is $III_{\lambda}$ where $\lambda = e^{-d_Gh(G)}$. In Section \ref{exits} we determine the
factor types of the extremal KMS weights for the gauge action coming from exits in the
graph, and they turn out to be of type $I_{\infty}$ when the
associated measure on the space of infinite paths in the graph is
atomic, and of type $II_{\infty}$ when it is not atomic. 

Like the results, the methods used in Section \ref{conservative} and
Section \ref{exits} are very different. Furthermore, the methods extend
to other actions, and in Section \ref{exits}
we can in fact handle all generalized gauge actions. The key lemma
here is valid for all these actions and
shows how the factor arising from an exit in the graph is closely
related to an Araki-Woods factor naturally defined from the exit. The method and results of Section
\ref{conservative} extend also well beyond the gauge action, and
we pursue this higher generality because it appears that the
variation of factor types becomes much richer for more general
actions. %Specifically, the results allow us to exhibit an example of a strongly
%connected row-finite recurrent graph $G$ with a family of
%generalised gauge actions on $C^*(G)$, one for each pair of positive
%real numbers
%$a_1 > a_2 $, for which the extremal KMS weights exhibit
%the following factor types: There is a positive real number $\beta_0$,
%depending on $a_1$ and $a_2$, and for each $\beta \geq \beta_0$ an essentially unique $\beta$-KMS
%weight of
%type $III_{\lambda(\beta)}$, where $
%\lambda(\beta) = e^{-\beta(a_1-a_2)}$ when $\beta > \beta_0$, while
%$\lambda(\beta_0) \in ]0,1]$ is determined by the condition that
%$$\lambda(\beta_0) = 1$$ when 
%$\overline{a_1\mathbb Z + a_2\mathbb Z} =
%\mathbb R$, and
%$$
%\mathbb Z \log \lambda(\beta_0) = \beta_0 a_1 \mathbb Z  + \beta_0
%a_2\mathbb Z
%$$
%when $\overline{a_1\mathbb Z + a_2\mathbb Z} \neq
%\mathbb R$. By cutting down to a corner of
%$C^*(G)$ all weights can be normalised to states and we obtain the
%same variation of factor types realised by KMS states on the
%corner. This variation of factor types resembles that in Example 5.3
%of \cite{Th3}, except that there is a possibility of having an abrupt change
%at the lowest value of $\beta$. 

In the last section we describe the ground states of a generalized gauge action
restricted to a corner defined from a vertex in the graph which no
longer needs to be strongly connected. %In
%the spirit of the first sections we consider also here generalized
%gauge actions, and the results we obtain hold for actions
%considerably more general than the gauge action. 
The method
is an adaptation of a method from \cite{Th2}, and as often before (e.g.
in \cite{LR}, \cite{Th2}, \cite{CL} and \cite{Ka}) the
ground states turn out to be parametrised by the states of a
$C^*$-algebra which in this case is a sub-quotient of the graph $C^*$-algebra
determined by the data used to define the one-parameter group. Its
structure can be quite complicated. We exhibit an example of a strongly connected row-finite
graph where the ground states of the gauge
action restricted to a corner are parametrised by the state space of
the CAR algebra. The ground states for generalized gauge actions on the graph
$C^*$-algebra itself were identified, under mild conditions, by Carlsen and Larsen in
\cite{CL}. There are considerably fewer of those, and in fact none unless the
graph has sinks or infinite emitters; at least not when the function
on the edges defining the action is strictly positive. In contrast, on a corner of the
graph $C^*$-algebra given by a vertex in a strongly connected infinite
graph,
there are always at least one ground state for the gauge action.

\section{Preparations}

Let $G$ be a directed graph with vertexes $V$ and edges $E$. We assume that $G$ is countable in the sense that $V$
and $E$ are both countable sets. We let $r$ and $s$ denote the maps
$r: E \to V$ and $s: E \to V$ which associate to an edge $e\in E$ its
target vertex $r(e)$ and source vertex $s(e)$, respectively. A vertex
$v$ is an
\emph{infinite emitter} when $s^{-1}(v)$ contains infinitely many
edges and a \emph{sink} when $s^{-1}(v)$ is empty. The union of
sinks and infinite emitters constitute a set which will be denoted by
$V_{\infty}$. Except for
the last section on ground states we will only consider graphs that are \emph{
strongly connected} in the sense that for all vertexes $v,w \in V$ there is a finite path
in $G$ which starts a $v$ and ends at $w$. In particular, there are no
sinks. The graph $C^*$-algebra
$C^*(G)$ is by definition the universal
$C^*$-algebra generated by a collection $S_e, e \in E$, of partial
isometries and a collection $P_v, v \in V$, of orthogonal projections subject
to the conditions that
\begin{enumerate}
\item[1)] $S^*_eS_e = P_{r(e)}, \ \forall e \in E$,
%\item[2)] $S_eS_e^* \leq P_{s(e),} \ \forall e \in E$,
\item[2)] $\sum_{e \in F} S_eS_e^* \leq P_v$ for every finite subset
  $F \subseteq s^{-1}(v)$ and all $v\in V$, and
\item[3)] $P_v = \sum_{e  \in s^{-1}(v)} S_eS_e^*, \ \forall v \in V
  \backslash V_{\infty}$.
\end{enumerate} 
The generalised gauge actions we consider can be defined using the
universal property. Let $F : E \to \mathbb R$ be any map. There is a
one-parameter group $\alpha^F = \left(\alpha^F_t\right)_{t \in \mathbb
  R}$ on $C^*(G)$ defined by the requirements
that
$$
\alpha^F_t(P_v) = P_v \ \text{and} \ \alpha^F_t(S_e) = e^{iF(e)t}S_e
$$
for all vertexes $v$, all edges $e$ and all real numbers $t$.

For the present purpose it is crucial that $C^*(G)$ can also be
realised as the (reduced) $C^*$-algebra of an \'etale groupoid
$\mathcal G$ constructed by A. Paterson in \cite{Pa}. Let $P_f(G)$ and
$P(G)$ denote the set of finite and infinite paths in $G$,
respectively. The range and source maps, $r$ and $s$, extend in the
natural way to $P_f(G)$; the source map also to $P(G)$. A vertex $v
\in V$ will be considered as a finite path of length $0$ and we set
$r(v) = s(v) =v$ when $v$ is considered as an element of $P_f(G)$. The unit space
$\Omega_G$ of $\mathcal G$ is the union
$\Omega_G = P(G) \cup Q(G)$,
where 
$$
Q(G) = \left\{p \in P_f(G): \ r(p) \in V_{\infty} \right\} 
$$ 
is the set of finite paths that terminate at a vertex
in $V_{\infty}$. In particular, $V_{\infty} \subseteq Q(G)$ because
vertexes are considered to be finite paths of length $0$. For any $p
\in P_f(G)$, let $|p|$ denote the length of $p$. When $|p| \geq 1$, set
$$
Z(p) = \left\{ q \in \Omega_G: \ |q| \geq |p| , \ q_i = p_i, \ i = 1,2,
  \cdots, |p| \right\},
$$
and
$$
Z(v) = \left\{ q \in \Omega_G : \ s(q) = v\right\}
$$
when $v \in V$. When $\nu \in P_f(G)$ and $F$ is a finite subset of $P_f(G)$, set
\begin{equation}\label{a6}
Z_F(\nu) = Z(\nu) \backslash \left(\bigcup_{\mu \in F} Z(\mu)\right) .
\end{equation}
The sets $Z_F(\nu)$ form a basis of compact and open subsets for a locally compact Hausdorff
topology on $\Omega_G$. When $\mu \in P_f(G)$ and $  x \in \Omega_G$, we can define the
concatenation $\mu x \in \Omega_G $ in the obvious way when $r(\mu) =
s(x)$. The groupoid $\mathcal G$ consists of the
elements in $\Omega_G \times \mathbb Z \times \Omega_G$ of the form
$$
(\mu x, |\mu| - |\mu'|, \mu'x),
$$
for some $x\in \Omega_G$ and some $ \mu,\mu' \in P_f(G)$. The product
in $\mathcal G$ is defined by
$$
(\mu x, |\mu| - |\mu'|, \mu' x)(\nu y, |\nu| -|\nu'|, \nu' y) = (\mu
x, \ |\mu | + |\nu| - |\mu'| - |\nu'|, \nu' y),
$$ 
when $\mu' x = \nu y$, and the involution by $(\mu x, |\mu| - |\mu'|,
\mu'x)^{-1} = (\mu' x, |\mu'| - |\mu|, \mu x)$. To describe the
topology on $\mathcal G$, let $Z_{F}(\mu)$ and $Z_{F'}(\mu')$ be two
sets of the form (\ref{a6}) with $r(\mu) = r(\mu')$. The topology we
shall consider has as a basis the sets of the form
\begin{equation}\label{top}
\left\{ (\mu x, |\mu| - |\mu'|, \mu' x) : \ \mu x \in Z_F(\mu), \
  \mu'x \in Z_{F'}(\mu') \right\} .
\end{equation}
With this topology $\mathcal G$ becomes an \'etale locally compact Hausdorff groupoid and we can consider the reduced $C^*$-algebra $C^*_r(\mathcal
G)$ as in
\cite{Re1}. As shown by Paterson in
\cite{Pa} there is an isomorphism $C^*(G) \to
  C^*_r(\mathcal G)$ which sends $S_e$ to $1_e$, where $1_e$ is the
  characteristic function of the compact and open set
$$
\left\{ (ex, 1, r(e)x) : \ x \in \Omega_G \right\} \ \subseteq \
\mathcal G,
$$  
and $P_v$ to $1_v$, where $1_v$ is the characteristic function of the
compact and open set
$$
\left\{ (vx,0,vx) \ : \ x \in \Omega_G \right\} \ \subseteq \ \mathcal G.
$$ 
In the following we use the identification $C^*(G) = C_r^*(\mathcal
G)$ and identify $\Omega_G$ with the unit space of $\mathcal G$ via
the embedding
$\Omega_G \ni x \ \mapsto \ ( x, 0,x)$.

By describing the generalised gauge action $\alpha^F$ in the groupoid
picture, it is seen that it is a special case of
actions considered by Renault in \cite{Re1}. Specifically we extend $F$ to a function
$F : P_f(G) \to \mathbb R$ such that 
$$
F(p_1p_2\cdots p_n) = \sum_{i=1}^n F(p_i)
$$
when $p = p_1p_2 \cdots p_n$ is a path of length $n \geq 1$ in $G$, and
$F(v) = 0$ when $v \in V$. We can
then define a continuous function $c_{F} : \mathcal G \to \mathbb R$
such that
$$
c_{F}(ux,|u|-|u'|,u'x) = F(u)- F(u').
$$
Since $c_{F}$ is a continuous homomorphism it gives rise to a continuous
one-parameter automorphism group $\alpha^F$ on $C^*_r(\mathcal G)$
defined such that
$$
\alpha^{F}_t(f)(\gamma) = e^{it c_{F}(\gamma)} f(\gamma)
$$
when $f \in C_c(\mathcal G)$, cf. \cite{Re1}. When $F$ is constant
$1$ this action is known as \emph{the gauge action} on $C^*(G)$ and we
denote it by $\gamma$.

Recall, \cite{KV1}, \cite{Th3}, that a weight $\psi$ on
the $C^*$-algebra $A$ is \emph{proper} when it is non-zero, densely defined and lower
semi-continuous. For such a weight, set
 $$
\mathcal N_{\psi} = \left\{ a \in A: \ \psi(a^*a) < \infty
\right\}.
$$ 
%Then 
%\begin{equation*}\label{f3}
%\mathcal N_{\psi}^*\mathcal N_{\psi} = \Span \left\{ a^*b : \ a,b \in
%  \mathcal N_{\psi} \right\}
%\end{equation*} 
%is a dense
%$*$-subalgebra of $A$, and there is a unique well-defined linear
%map $\mathcal N_{\psi}^*\mathcal N_{\psi} \to \mathbb C$ which agrees
%with $\psi$ on positive elements.
%extends $\psi : \mathcal N_{\psi}^*\mathcal N_{\psi} \cap A_+ \to
%[0,\infty)$.
%$$
%semi-continuous 
Let $\alpha : \mathbb R \to \Aut A$ be a point-wise
norm-continuous one-parameter group of automorphisms on
$A$. Let $\beta \in \mathbb R$. Following Combes, \cite{Co}, we say that a proper weight
$\psi$ on $A$ is a \emph{$\beta$-KMS
  weight} for $\alpha$ when
\begin{enumerate}
\item[i)] $\psi \circ \alpha_t = \psi$ for all $t \in \mathbb R$, and
\item[ii)] for every pair $a,b \in \mathcal N_{\psi} \cap \mathcal
  N_{\psi}^*$ there is a continuous and bounded function $F$ defined on
  the closed strip $D_{\beta}$ in $\mathbb C$ consisting of the numbers $z \in \mathbb C$
  whose imaginary part lies between $0$ and $\beta$, and is
  holomorphic in the interior of the strip and satisfies that
$$
F(t) = \psi(a\alpha_t(b)), \ F(t+i\beta) = \psi(\alpha_t(b)a)
$$
for all $t \in \mathbb R$. \footnote{As in \cite{Th1} we apply the
  definition from \cite{C} for the action $\alpha_{-t}$
  in order to use the same sign convention as in \cite{BR}, for example.}
\end{enumerate}  
A $\beta$-KMS weight $\psi$ on $A$ is \emph{extremal} when the only
$\beta$-KMS weights $\varphi$ on $A$ with the property that
$\varphi(a^*a) \leq \psi(a^*a)$ for all $a \in A$ are the scalar
multiples of $\psi$, viz. $\varphi = t\psi$ for some $t > 0$.

It was shown in \cite{Th3} that all gauge invariant KMS weights $\psi$
for $\alpha^F$ are given by a
regular Borel measure $m$ on $\Omega_G$ in the sense that
\begin{equation}\label{measure}
\psi(a) = \int_{\Omega_G} P(a) \ dm,
\end{equation}
where $P : C^*(\mathcal G) \to C_0(\Omega_G)$ is the canonical
conditional expectation, \cite{Re1}. Under a certain condition on $F$,
spelled out in (\ref{nov30}) below, the KMS weights for
$\alpha^F$ are automatically gauge invariant and hence they all come
from measures on $\Omega_G$. A measure $m$ corresponding to a
$\beta$-KMS measure via (\ref{measure}) will be called a
\emph{$\beta$-KMS measure}. The measure associated to an extremal KMS
weight is either supported on $P(G)$, in which case it is
\emph{harmonic}, or on $Q(G)$ and it is then a \emph{boundary} KMS
weight. See \cite{CL} and \cite{Th5} for more on this dichotomy.

Given a weight $\psi$ on a $C^*$-algebra $A$ there is a GNS-type
construction consisting of a Hilbert space $H_{\psi}$, a linear map
$\Lambda_{\psi} : \mathcal N_{\psi} \to H_{\psi}$ with dense range and
a non-degenerate representation $\pi_{\psi}$ of $A$ on $H_{\psi}$ such that
\begin{enumerate}
\item[$\bullet$] $\psi(b^*a) = \left<
    \Lambda_{\psi}(a),\Lambda_{\psi}(b)\right>, \ a,b \in \mathcal
  N_{\psi}$, and
\item[$\bullet$] $\pi_{\psi}(a)\Lambda_{\psi}(b) = \Lambda_{\psi}(ab), \
  a \in A, \ b \in \mathcal N_{\psi}$,
\end{enumerate} 
cf. \cite{KV1},\cite{KV2}.
 As
observed in Lemma 4.4 in \cite{Th3} the von Neumann algebra
$\pi_{\psi}(A)''$ is a factor when $\psi$ is extremal, and the von
Neumann algebra type of $\pi_{\psi}(A)''$ is the \emph{factor type} of
$\psi$.

 As mentioned in the introduction, the extremal boundary KMS
weights of a generalized gauge action on a graph $C^*$-algebra are of factor type
$I_{\infty}$, at least when the graph is strongly connected. This can be seen for example by showing that
the factor in question contains an abelian projection and can not be
finite, in much the same way as in the proof of Proposition 5.1 in
\cite{Th2}. We leave the details to the reader and focus here on the extremal harmonic KMS weights.

\section{The conservative KMS weight}\label{conservative}

When the graph $G$ is strongly connected and recurrent, and the
Gurevich entropy $h(G)$ of $G$ is finite, there is a $h(G)$-KMS weight $\psi$
for the gauge action on $C^*(G)$ which is unique up to multiplication
by scalars, cf. \cite{EFW} for the finite case and Proposition 4.16 in
\cite{Th5} for the infinite case. In particular $\psi$ is extremal and in this
section we determine its factor type by using the method from
Section 4.2 of \cite{Th3}. The conservative nature on the measure on
$P(G)$ given by the weight is the crucial additional information which
will allow us to be more conclusive than in \cite{Th3}.

 Let $\psi$ be an extremal $\beta$-KMS weight for
$\alpha^F$. It follows from Section 2.2 in \cite{KV1} that $\psi$ extends to a normal semi-finite faithful weight
$\tilde{\psi}$ on $M = \pi_{\psi}\left(C^*(G)\right)''$ such that $\psi =
\tilde{\psi} \circ \pi_{\psi}$, and that the modular group on
$M$ corresponding to $\tilde{\psi}$ is the
one-parameter group $\theta$ given by
$$
\theta_t = \tilde{\alpha}^F_{-\beta t},
$$
where $\tilde{\alpha}^F$ is the
$\sigma$-weakly continuous extension of $\alpha^F$ defined such that
$\tilde{\alpha}^F_t \circ \pi_{\psi} = \pi_{\psi} \circ {\alpha}^F_t$. To
simplify the notation in the following, let $N \subseteq M$ be the fixed point
algebra of $\theta$, viz. $N = M^{\theta}$, and suppress the
representation $\pi_{\psi}$ so that, in particular, $1_v \in
C_c(\mathcal G)$ will now also denote the projection
$\pi_{\psi}(1_v)\in M$. For $f \in L^1(\mathbb R)$ define a linear map
$\sigma_f : M \to M$ such that
$$
\sigma_f(a) = \int_{\mathbb R} f(t) \theta_t(a) \ dt.
$$  
For every vertex $v \in V$ and every central projection $e$ in
$1_vN1_v$, set
$$
\Sp (eMe) = \bigcap \left\{ Z(f) : \ f \in L^1(\mathbb R), \
  \sigma_f(eMe) = \{0\} \right\} ,
$$
where
$$
Z(f) = \left\{ r \in \mathbb R : \ \int_{\mathbb R} e^{itr} f(t) \ dt
  \ = \ 0 \right\} .
$$
Then the invariant $\Gamma(M)$ introduced by Connes, \cite{C}, can be
expressed as the intersection
$$
\Gamma(M) = \bigcap_{e} \Sp (eMe),
$$
where we take the intersection over
all non-zero central projections $e$ in $1_vN1_v$. In particular, the
intersection does not depend on which vertex $v$ we use.

Pick a vertex $v$ in $G$ and note that  
$$
 \left\{ \beta F(\mu) - \beta F(\mu'): \ \mu, \mu' \in P_f(G), \ r(\mu) =
  r(\mu') = s(\mu) = s(\mu') = v \right\} 
$$
is a subgroup of $\mathbb R$ which does not depend on the vertex $v$
since $G$ is strongly connected. Let $R_{G,F}$ be the closure in $\mathbb
R$ of this subgroup.

\begin{lemma}\label{nov21a} Assume that $G$ is strongly connected. Let $\psi$ be an extremal $\beta$-KMS
  weight for $\alpha^F$. Then $ \pi_{\psi}(C^*(G))''$ is a hyperfinite
  factor and
\begin{equation}\label{nov21}
\Gamma\left( \pi_{\psi}(C^*(G))''\right) \subseteq R_{G,F}.
\end{equation}

\end{lemma}
\begin{proof}  $M = \pi_{\psi}(C^*(G))''$ is hyperfinite because $C^*(G)$
  is nuclear. We show that 
$$
\mathbb R \backslash R_{G,F} \ \subseteq \ \mathbb R \backslash
\Gamma(M).
$$
Let $r \in \mathbb R \backslash R_{G,F}$ and choose a function $f \in
L^1(\mathbb R)$ whose Fourier transform $\hat{f}$ satisfies that
$\hat{f}(t) = 0$ for all $t \in R_{G,F}$ and $\hat{f}(r) \neq
0$. Assume that
$h \in C_c(\mathcal G)$ is supported on a set of the form
\begin{equation}\label{nov32a}
\left\{(\mu x, |\mu|-|\mu'|,\mu'x) : \ x \in \Omega_G \right\} ,
\end{equation}
where $\mu, \mu' \in P_f(G)$ satisfy that $s(\mu) = s(\mu') = v$ and
$r(\mu) = r(\mu')$. Then
$$
\sigma_f(h) = \int_{\mathbb R} f(t) \theta_t(h) \ dt =
\hat{f}(\beta(F(\mu') - F(\mu)) h = 0
$$
because $\beta(F(\mu') - F(\mu)) \in R_{G,F}$. Note that elements of
the form $1_vh1_v$ span a strongly dense subspace of $1_vM1_v$ and
conclude that $\sigma_f(1_vM1_v) = \{0\}$. Since $\hat{f}(r) \neq 0$
we conclude that $r \notin \Gamma(M)$.

\end{proof}

%The following theorem generalizes a result from \cite{EFW}. 

As in \cite{Th5} we introduce the matrix $A(\beta)$ over $V$ such that
\begin{equation}\label{matrix}
A(\beta)_{vw} \ = \  \sum_{e \in s^{-1}(v) \cap r^{-1}(w)} e^{-\beta
  F(e)} .
\end{equation}
It was shown in \cite{Th5} that there are no gauge invariant $\beta$-KMS weights for
$\alpha^F$ unless all powers of $A(\beta)$ are finite and $\limsup_n
  \left(A(\beta)^n_{vv}\right)^{\frac{1}{n}} \leq 1$ for one and
  hence all vertexes $v \in V$. When these conditions hold and
  $A(\beta)$ is $1$-recurrent, in the sense that
$$
\sum_{n=0}^{\infty} A(\beta)^n_{vv} = \infty ,
$$
there is a unique ray of gauge invariant $\beta$-KMS weights for the action
$\alpha^F$. For $G$ finite this is a result of Enomoto, Fujii and
Watatani, \cite{EFW}, when $F =1$, and a result of Exel and Lace when
$F$ is either strictly positive or strictly negative, \cite{EL}. The
general case follows from Theorem 4.14 and
Proposition 4.16 in \cite{Th5}. Furthermore, it was shown that these KMS weights are harmonic and that the
associated measure on $P(G)$ is conservative. The
following result determines their factor type under a
certain condition on $F$.

\begin{thm}\label{nov9} Assume that $G$ is strongly connected and that
  $F$ satisfies the following condition:
\begin{equation}\label{nov30}
\nu \in P_f(G), \ |\nu| \geq 1, \ s(\nu) = r(\nu) \ \Rightarrow \
F(\nu) \neq 0 .
\end{equation}
Let $\psi$ be a $\beta$-KMS weight
  for $\alpha^F$, $\beta \neq 0$, whose associated measure is
  supported and
  conservative on $P(G)$. Then $ \pi_{\psi}(C^*(G))''$ is a hyperfinite factor and
$$
\Gamma( \pi_{\psi}(C^*(G))'') = R_{G,F} .
$$ 
\end{thm}

\begin{proof} Note, first of all, that under the present assumptions the
  real-valued homomorphism $c_F$ on $\mathcal G$ defined from $F$ satisfies
  the two conditions on $c_0$ in Theorem 2.2 of \cite{Th3}. This implies that all KMS weights of $\alpha^F$ are
  gauge invariant. Furthermore, since the measure associated with
  $\psi$ is supported and conservative on $P(G)$, it follows from Theorem 4.11, Theorem 4.14 and Proposition 4.16 in
  \cite{Th5} that up to multiplication by scalars $\psi$ is the only
  $\beta$-KMS weight for $\alpha^F$. In particular, $\psi$ is extremal and $ M =
  \pi_{\psi}(C^*(G))''$ is a factor.

By Lemma \ref{nov21} and the discussion preceding it, it suffices here
to consider a non-zero central projection $q
\in 1_vN1_v$ for some vertex $v\in V$, and show that $R_{G,F} 
\subseteq \Sp(qMq)$. To this end consider two finite paths $l^{\pm}$
in $G$ such
that $s(l^{\pm}) =  r(l^{\pm}) = v$. It suffices to show that
$\beta (F(l^+)-F(l^-)) \in \Sp (qMq)$. Let $\omega$ be the state on $1_vM1_v$ given by $\omega(a) =
\psi(1_v)^{-1}\tilde{\psi}(a)$. Then $\omega$ is a faithful normal state which
is a trace on $1_vN1_v$
and we consider the corresponding $2$-norm 
$$
\|a\|_v = \sqrt{\omega(a^*a)} .
$$
By Kaplansky's density theorem there is an element $f \in 1_vC_c(\mathcal G)1_v$
  such that $0 \leq f \leq 1$ and
  $\left\|q- f\right\|_v$ is as small as we want. Since $f$ can be
  approximated in norm by a linear combination of
  characteristic functions of sets of the form (\ref{nov32a}) we may as
  well assume that $f$ is such a linear combination. Then the limit
$$
g = \lim_{R \to \infty}  \frac{1}{R} \int_0^R \alpha^F_{t}(f) \ dt \ 
$$
exists in norm, and $g \in 1_v C_c(\mathcal G)1_v$  is supported on the open
sub groupoid $\mathcal R$ of $\mathcal G$ consisting of the elements
$(\mu x, |\mu| -|\mu'|,\mu' x)$ with $F(\mu) = F(\mu')$. Note that
$$
\left\| q -   \frac{1}{R} \int_0^R \alpha^F_{t}(f) \ dt 
\right\|_v  = \left\|\frac{1}{R} \int_0^R \tilde{\alpha}^F_{t}(q -f) \ dt 
\right\|_v \leq \left\| q-f\right\|_v
$$
by Kadisons Schwarz inequality. We may therefore assume that $f \in 1_vC_c(\mathcal R)1_v$. The groupoid $\mathcal R$ has trivial
isotropy groups since we
assume that (\ref{nov30}) holds, and $\mathcal R$ is therefore an \'etale equivalence relation. It follows therefore from (the
  proof of) Lemma 2.24 in \cite{Th1} that there is a sequence 
$$
\{d^n_j : \ j = 1,2,\cdots, N_n\}, \ n = 1,2, \cdots
$$
of non-negative elements in $C(Z(v))$ such that
$\sum_{j=1}^{N_n} {d^n_j}^2 = 1_v$ and 
$$
P(f) = \lim_{n \to \infty} \sum_{j=1}^{N_n} d^n_j fd^n_j ,
$$ 
where $P$ is the conditional expectation $P : 1_vC^*_r(\mathcal R)1_v
\to C(Z(v))$ and the limit is norm-convergent. Note that 
$$
q =  \sum_{j=1}^{N_n} d^n_j qd^n_j
$$
for all $n$ because $q$ is central in
$1_vN1_v$, and that
%$$
%\left( \sum_{j=1}^{N_n} d^n_j (q-f)d^n_j\right)^*\left( \sum_{j=1}^{N_n}
%  d^n_j (q-f)d^n_j\right) \leq \sum_{j=1}^{N_n} d^n_j (q-f)^*(q-f)d^n_j
%$$
%by the Cauchy-Schwarz inequality for completely positive maps. It
%follows that 
$$
\left\|q-P(f)\right\|_v = \lim_{n \to \infty} \left\| \sum_{j=1}^{N_n}
  d^n_j(q-f)d^n_j\right\|_v
\leq \left\|q-f\right\|_v .
$$ 
Using $P(f)$ in the place of $f$ we may
therefore assume that $f \in C(Z(v))$ and that there
are finite paths $\nu_i \in P_f(G)$, finite collections of paths $F_i
\subseteq P_f(G)$ and numbers $0 < t_i \leq 1, \ i = 1,2,\cdots,N$,
such that $Z_{F_i}(\nu_i) \cap Z_{F_j}(\nu_j) = \emptyset$ when $i
\neq j$ and
%$$
%\left\| f - \sum_{i=1}^N t_i 1_{Z_{F_i}(\nu_i)}\right\|
%$$
%is as small as we want. We may therefore assume that
$$
f =  \sum_{i=1}^N t_i 1_{Z_{F_i}(\nu_i)} .
$$
Fix $i$ for a moment. There is a finite or countably infinite
set $H'_i \subseteq P_f(G)$ such that
$$
Z_{F_i}(\nu_i) = \sqcup_{\nu \in H'_i} Z(\nu) .
$$ 
It follows from (\ref{measure}) that there is a Borel probability measure $m$
on $Z(v)$ such that 
$$
\omega(g) = \int_{Z(v)} P(g) \ dm
$$
when $g \in 1_vC^*(G)1_v$. By Theorem 4.11 in \cite{Th5} $m$ is
supported on
$$
\left\{(x_i)_{i=1}^{\infty} \in P(G): \ s(x_1) = s(x_i) = v \ \text{for
    infinitely many} \ i \right\}
$$
and we can therefore choose a finite set $H_i \subseteq P_f(G)$ such
that $r(\nu) = s(\nu) = v$ for all $\nu \in H_i$,
$$
 \sqcup_{\nu \in H_i} Z(\nu) \subseteq Z_{F_i}(\nu_i)
$$
and
$$
\left\| 1_{Z_{F_i}(\nu_i)} - \sum_{\nu \in H_i} 1_{Z(\nu)}\right\|_v
= \sqrt{m\left( Z_{F_i}(\nu_i) \backslash \sqcup_{\nu \in H_i} Z(\nu) \right)}
$$
is as small as we want. By exchanging $\sum_{\nu \in H_i} 1_{Z(\nu)}$
for $ 1_{Z_{F_i}}(\nu_i)$ for each $i$, we can therefore assume that there
are finite paths $\nu_i \in P_f(G)$ and numbers $0 \leq t_i \leq 1, \
i = 1,2,\cdots,M'$, such that $s(\nu_i) = r(\nu_i) = v$ for all $i$,
$$
f =  \sum_{i=1}^{M'} t_i 1_{Z(\nu_i)} ,
$$
and $Z(\nu_i) \cap Z(\nu_j) = \emptyset$ when $i \neq j$.
Finally, since $q$ is a projection a standard argument, as in the
proof of Lemma 12.2.3 in \cite{KR}, allows us select a subset of the
$\nu_i$'s to arrange, after a renumbering, that 
\begin{equation}\label{nov17}
p = \sum_{i=1}^M 1_{Z(\nu_i)}
\end{equation}
is a projection in $C(Z(v))$ such that
\begin{equation}\label{nov8}
\left\|q - p\right\|_v \leq \epsilon ,
\end{equation}
where $\epsilon > 0$ is as small as we need. We
choose $\epsilon > 0$ so small that
$$
e^{-F(l^+)\beta}
\epsilon  +e^{-(F(l^+)-F(l^-))\beta} \epsilon + 2\epsilon  \ <  \ e^{-F(l^+)\beta}
\omega(q).
$$
For each $\nu_i$ from (\ref{nov17}) we let $w^{\pm}_i \in 1_vC_c(\mathcal G)1_v$ be the
characteristic function of the compact and open set
$$
\left\{ (\nu_ix, - |l^{\pm}|, \nu_il^{\pm}x) : \ x \in \Omega_G
\right\}
$$
in $\mathcal G$. Each $w^{\pm}_i$ is a partial isometry such that
\begin{enumerate}
\item[a)] $w^{\pm}_i{w^{\pm}_i}^* = 1_{Z(\nu_i)}$ ,
\item[b)] ${w^{\pm}_i}^* w^{\pm}_i = 1_{Z(\nu_il^{\pm})} \leq
  1_{Z(\nu_i)}$, and
\item[c)] $\alpha^F_t\left(w^{\pm}_i\right) = e^{-i F(l^{\pm})t}w^{\pm}_i$
  for all $t \in \mathbb R$.
\end{enumerate} 
It follows from Lemma 2.9 in
\cite{Th5} that
\begin{enumerate}
\item[d)] $m(Z(\nu_il^{\pm})) = e^{-\beta F(l^{\pm})}m(Z(\nu_i))$ for all $i$.
\end{enumerate}
Set $w_{\pm} = \sum_{i=1}^M w^{\pm}_i$ and note that $w_{\pm}$ are
both partial isometries. It follows from b) that $w_+p = w_+$ and therefore from
(\ref{nov8}) that
\begin{equation*}
\begin{split}
& \omega (q{w_+}^*w_-q{w_-}^*w_+q) \geq
\omega({w_+}^*w_-q{w_-}^*w_+) - 2\epsilon .
\end{split}
\end{equation*}
Since $\psi$ is a $\beta$-KMS weight, it follows from c) that
$$
\omega({w_+}^*w_-q{w_-}^*w_+) = e^{-(F(l^+)- F(l^-))\beta}
\omega(q{w_-}^*w_+{w_+}^*w_-), 
$$
which thanks to a) is the same as
$$
 e^{-(F(l^+)-F(l^-))\beta}
\omega(q{w_-}^*w_-) .
$$
Using (\ref{nov8}) again we find that
$$
\omega(q{w_-}^*w_-) \geq \omega(p{w_-}^*w_-) - \epsilon .
$$
It follows from b) that $ \omega(p{w_-}^*w_-) =
\omega({w_-}^*w_-)$ while b), d) and a) imply that
$$
\omega({w_-}^*w_-) = m\left(\cup_i Z(\nu_i l^-)\right) = e^{-F(l^-)\beta} \omega(p). 
$$
Since $\omega(p) \geq \omega(q) - \epsilon$ by (\ref{nov8}) we can
put everything together and conclude that
$$
\omega(q{w_+}^*w_-q{w_-}^*w_+q) \geq  e^{-F(l^+)\beta}
\omega(q) - e^{-F(l^+)\beta}
\epsilon  -e^{-(F(l^+)-F(l^-)) \beta} \epsilon -2\epsilon .
$$
It follows therefore from the choice of $\epsilon$ that
$\omega(q{w_+}^*w_-q{w_-}^*w_+q) > 0$. By c), 
$$
\tilde{\alpha}^F_t(q{w_+}^*w_-q) = e^{i(F(l^+) - F(l^-))t} q{w_+}^*w_-q
$$
and hence $\theta_{t} (q{w_+}^*w_-q) = e^{-it\beta( F(l^+) -F(l^-))}
q{w_+}^*w_-q$ for all $t \in \mathbb R$. Since $q{w_+}^*w_-q \neq 0$
it follows now from Lemme 2.3.6 in \cite{C} that $\beta( F(l^+) -F(l^-)) \in \Sp(qMq)$, as required.
  
\end{proof}

 A part of the assumption in Theorem
  \ref{nov9} is that $\beta \neq 0$, and both the proof and the
  conclusion fails in general if $\beta =  0$. However, a $0$-KMS weight is a
  densely defined trace, and it can not exist when $C^*(G)$ is simple
  and purely infinite, which it is when $G$ is strongly connected
  unless $G$ only consists of
  a single finite loop. This is therefore the only case ruled out by
  the assumption $\beta \neq 0$.

By using Theorem \ref{nov9} the factor type of the KMS weight $\psi$
can be determined from the rule summarised by
Pedersen in 8.15.11 of \cite{Pe}. When $G$ is finite Theorem \ref{nov9} recovers the results of Okayasu
from \cite{O}. As pointed in \cite{Th5} there can
be KMS states on $C^*(G)$ when $G$ is finite, also for actions that are not
covered by Okayasu's work. In such cases the factor type can be
determined from
Theorem \ref{nov9} because the condition (\ref{nov30}) is satisfied,
cf. Proposition 4.15 in \cite{Th5}.

To
formulate the result for the gauge action, let the \emph{global
  period} $d_G$ of $G$ be the greatest common divisor of the set
$$
\left\{ |\nu| \geq 1 : \ \nu \in P_f(G), \ s(\nu) = r(\nu) = v
\right\}
$$
for some vertex $v$; the number is independent of which vertex we
choose. Since $G$ is strongly connected the global period is the same
as the $d_G$ introduced for arbitrary graphs in Section 4.2 of \cite{Th3}.

\begin{cor}\label{nov14} 
 Assume that $G$ is strongly connected, recurrent and with finite Gurevich
  entropy $h(G) > 0$. Let $\psi$ be the essentially unique $h(G)$-KMS weight for the gauge
  action  on $C^*(G)$. The
  factor type of $\psi$ is $III_{\lambda}$ where $\lambda = e^{-d_Gh(G)}$.  
\end{cor}
\begin{proof} $\psi$ exists under the given assumptions and the
  corresponding measure is conservative by \cite{Th5}. Hence Theorem
  \ref{nov9} applies and the result follows because
$$
R_{G,F} = d_G h(G) \mathbb Z 
$$
since $F =1$.
\end{proof}

The condition $h(G) > 0$ rules out only the case where $G$ is a single
finite loop.

\section{The KMS weights from exits}\label{exits}

We retain in this section the assumption that $G$ is strongly
connected. As in \cite{Th5} an \emph{exit path} in $G$ is a sequence
$t=(t_i)_{i=1}^{\infty}$ of vertexes such that there is an edge $e_i$ with
$s(e_i) = t_i$ and $r(e_i) = t_{i+1}$ for all $i$, and such that $\lim_{i
  \to  \infty} t_i = \infty$ in the sense that $t_i$ eventually leave
any finite subset of vertexes. For a given exit path
$t$ and a real number $\beta$ we set
$$
t^{\beta}(i) = A(\beta)_{t_1t_2} A(\beta)_{t_2t_3}\cdots A(\beta)_{t_{i-1}t_i}
$$
where $A(\beta)$ is the matrix (\ref{matrix}). 
Then $t$ is \emph{$\beta$-summable} when the limit
$$
\lim_{i \to \infty} t^{\beta}(i)^{-1} \sum_{n=0}^{\infty} A(\beta)^n_{vt_i} 
$$
is finite for one, and hence for all vertexes $v \in V$. As shown in
\cite{Th5} a $\beta$-summable exit gives rise to a harmonic
$\beta$-KMS measure $m_t$ determined by the condition that
$$
m_t(Z(v)) = \lim_{i \to \infty} t^{\beta}(i)^{-1} \sum_{n=0}^{\infty}
A(\beta)^n_{vt_i} 
$$
for every vertex $v$. It follows from Corollary 5.4 in \cite{Th5} that
the corresponding KMS weight $\psi_t$ on $C^*(G)$ is extremal among the
gauge invariant KMS weights. But in fact, $\psi_t$ is extremal among
all KMS weights when $G$ is strongly connected. This follows from a
combination of Proposition 2.6 in \cite{Th5} with Theorem 1.3 in \cite{N}
by observing that the isotropy group in $\mathcal G$ of an element in
$P(G)$ which is not pre-periodic is trivial. We can therefore talk about
the factor type of $\psi_t$ without further assumptions on $F$. 

To determine the factor type of $\psi_t$, set 
$$
k_n = \# s^{-1}(t_n) \cap
r^{-1}(t_{n+1})
$$
which is finite, cf. \cite{Th5}. Choose a numbering $e^n_1,e^n_2, \cdots , e^n_{k_n}$ of the edges in
$s^{-1}(t_n) \cap r^{-1}(t_{n+1})$, and let $\omega_n$ be the state on $M_{k_n}(\mathbb C)$
given by
\begin{equation}\label{nov26}
\omega_n (a) = \frac{\Tr \left( e^{-\beta H}
    a\right)}{\Tr\left(e^{-\beta H}\right)},
\end{equation}
where $H = \diag\left(F(e^n_1),F(e^n_2), \cdots,
  F(e^n_{k_n})\right)$. Then $\omega = \otimes_{n=1}^{\infty}
\omega_n$ is a state on the UHF-algebra
\begin{equation}\label{UHF}
A = \otimes_{n=1}^{\infty} M_{k_n}(\mathbb C) ,
\end{equation}
and we set
$$
\mathcal R(t) = \pi_{\omega}(A)'' .
$$ 
Note that $\mathcal R(t)$ is an Araki-Woods factor, cf. \cite{AW}.

\begin{lemma}\label{nov24} Let $G$ be a strongly connected graph and $t$
  a $\beta$-summable exit path in $G$, and let $\psi_t$ be the
  corresponding $\beta$-KMS weight for the generalised gauge action
  $\alpha^F$ on
  $C^*(G)$. There is a projection $p \in  \pi_{\psi_t}(C^*(G))''$ in
  the fixed point algebra of $\tilde{\alpha}^F$ such
  that 
$$
p \pi_{\psi_t}(C^*(G))''p \simeq \mathcal R(t) .
$$ 
\end{lemma}
\begin{proof}
Set $M = \pi_{\psi_t}(C^*(G))''$. As in the proof of
  Theorem \ref{nov9} we suppress the representation $\pi_{\psi_t}$ in
  the notation. For each $n \in \mathbb N$ let 
$$
A_n = \left\{ \xi \in
    P_f(G): \ \xi = e_1e_2\cdots e_{n-1}, \ s(e_i) = t_i, \ i =1,2,
    \cdots, n-1, \ r(e_n) = t_n \right\}.
$$ 
When $\xi,\xi' \in A_n$,
  let $1_{\xi,\xi'} \in C_c(\mathcal G)$ be the characteristic
  function of 
$$
\left\{ (\xi x, 0,\xi'x) : \ x \in \Omega_G \right\} .
$$
Set $q_n = \sum_{\xi \in A_n} 1_{\xi, \xi}\in C_c(\Omega_G)$. Then $q_n \geq q_{n+1}$
and we define $p \in M$ to the limit 
$$
p  = \lim_{n \to\infty} q_n 
$$ 
in the strong operator topology. If we let $\tilde{\psi_t}$ denote the
normal weight on $M$ which extends $\psi_t$, we find that 
$$ 
\tilde{\psi_t}\left(p\right) = \lim_{n \to \infty}
\psi_t(q_n) = m_t(\pi^{-1}(t)) = 1,
$$
cf. Lemma 5.3 in \cite{Th5}.
In particular, $ p \neq 0$, and $\tilde{\psi_t}$ is a state on
$pMp$. Since $C_0(\Omega_G) \subseteq C^*(G)$ is in the fixed point
algebra of $\alpha^F$ it follows that $p$ is fixed by $\tilde{\alpha}^F$.
Set 
$$
p_{\xi,\xi'} = p1_{\xi,\xi'}p .
$$
A calculation in the $*$-algebra $C_c(\mathcal R)$ shows that 
$$
q_m1_{\xi_1,\xi_2}q_mq_m1_{\xi_3,\xi_4}q_m = \begin{cases} 0, &  \
  \text{when} \ \xi_2\neq \xi_3 \\ q_m1_{\xi_1,\xi_4}q_m, & \
\text{when} \ \xi_2 = \xi_3 \end{cases},
$$
provided $m > n$. It follows that  $\left\{ p_{\xi, \xi'} : \ \xi, \xi' \in A_n\right\}$ is
the set of matrix units in a unital $C^*$-subalgebra $\mathcal M_n$ of dimension $\#
A_n$ in $pMp$. Observe that
$$  
p_{\xi,\xi'} =
\sum_e p_{\xi e, \xi' e}
$$ 
when we sum over all edges $e \in s^{-1}(t_n) \cap
r^{-1}(t_{n+1})$. It follows that $\mathcal M_n \subseteq \mathcal
M_{n+1}$. Note also that $ pM p$ is the
closure in the strong operator topology of 
$$
 p1_{t_1}C_c(\mathcal G)1_{t_1}p .
$$
By definition of the topology of $\mathcal G$ any $f\in 1_{t_1}C_c(\mathcal
G)1_{t_1}$ can be approximated in norm by finite linear combinations of
functions that are the characteristic function of a set $B$ of the form
(\ref{top}) with $s(\mu) = s(\mu') = t_1$.
Note that
$q_m1_Bq_m = 0$
when $m > \max \{|\mu|,\mu'|\}$ and $|\mu| \neq |\mu'|$. So assume
that $|\mu| = |\mu'| =k$. Then $q_m 1_Bq_m = 0$ for $m > k$ unless
$\mu, \mu' \in A_{k}$. So assume that $\mu, \mu' \in A_k$. When $m$ is
larger than $m_0 = k+\max\{|\nu| : \ \nu \in F \cup F' \}$ we see that
there are subsets $F_1,F_2 \subseteq A_{m_0}$ such that
$$
q_m1_Bq_m = \sum_{\xi \in F_1,\xi' \in F_2} q_m 1_{\xi,\xi'}q_m .
$$
It follows first that $ p1_B p \in
\mathcal M_{m_0}$ and then that $\bigcup_n \mathcal
  M_n$ is dense in 
$p M p$ for the strong operator
topology.

Using the same numbering of the edges in $s^{-1}(t_n) \cap
r^{-1}(t_{n+1})$ as above we set
$$
u^n_t = \diag \left(e^{iF(e^n_1)t}, e^{iF(e^n_2)t}, \cdots ,
  e^{iF(e^n_{k_n})t)}\right) \ \in \ M_{k_n}(\mathbb C) .
$$
The norm closure $\overline{\bigcup_n \mathcal M_n}$ is a copy of
the UHF-algebra $A$ from (\ref{UHF})
and in this picture the restriction of the automorphism group
$\tilde{\alpha}^F$ to $\overline{\bigcup_n \mathcal M_n}$ is the tensor product action
$$
\otimes_{n=1}^{\infty} \Ad u^n_t, \ t \in \mathbb R .
$$
%For each $n$ we let $\omega_n$ be the state on $M_{k_n}(\mathbb C)$
%given by
%$$
%\omega_n (a) = \frac{\Tr \left( e^{-\beta H}
%    a\right)}{\Tr\left(e^{-\beta H}\right)},
%$$
%where $H = \diag\left(F(e^n_1),F(e^n_2), \cdots,
%  F(e^n_{k_n})\right)$. 
The tensor product state $\omega =
\otimes_{n=1}^{\infty} \omega_n$ is a $\beta$-KMS state for
$\otimes_{n=1}^{\infty} \Ad u^n_t$ on $A$, and the same is the restriction of
$\tilde{\psi_t}$ to $\overline{\bigcup_n \mathcal M_n}$. The two states
must therefore agree under the identification $\overline{\bigcup_n
  \mathcal M_n} = A$, thanks to uniqueness of the $\beta$-KMS state, \cite{K}. It follows that
$$
pMp \simeq \pi_{\omega}(A)'' = \mathcal R(t).
$$
%We can therefore now use Remarque 1.3.5 and Th\'er\`eme 1.3.7 in
%\cite{C} to conclude that
%$$
%T(M)= T(1_{\pi^{-1}(t)}M1_{\pi^{-1}(t)}) = T( \pi_{\omega}(A)'') 
%$$
%is the subgroup of $\mathbb R$ described in the theorem.
\end{proof}

We will say that the exit path $t = (t_n)_{n=1}^{\infty}$ is \emph{slim}
when $s^{-1}(t_i) \cap r^{-1}(t_{i+1})$ only contains one edge for all $i$ large enough.

\begin{thm}\label{nov25}  Let $G$ be a strongly connected graph and $t$
  an exit path in $G$ which is $\beta$-summable for the gauge action, and let $\psi_t$ be the
  corresponding $\beta$-KMS weight for the gauge action on
  $C^*(G)$. The factor type of $\psi_t$ is $I_{\infty}$
  when $t$ is slim and $II_{\infty}$ factor when it is
  not.
\end{thm}
\begin{proof} The gauge action arises by choosing $F$ to be constant
  $1$. It follows that $\mathcal R(t)$ is a type I factor when $t$ is
  slim and the hyper finite $II_1$ factor otherwise. It follows
  therefore from Lemma \ref{nov24} that $
  \pi_{\psi_t}(C^*(G))''$ is type $I$ when $t$ is slim and type $II$
  otherwise. To complete the proof we need to show that $M$ is not finite. To this
end choose for each $n \geq 2$ a path $\mu_n$ in $G$ such that
$s(\mu_n) = t_n$ and $r(\mu_n) = t_1$. The characteristic function
$V_n$ of
the set
$$
\left\{ (\mu_n x, |\mu_n|, x) : \ x \in Z(t_1) \right\}
$$
is an element of $C_c(\mathcal G)$ such that $V_n^*V_n = 1_{t_1}$ and
$V_nV_n^* \leq 1_{t_n}$. In particular, the $V_n$'s are non-zero
partial isometries in $M$ and since $t$ is an exit, an infinite
sub collection of them will have orthogonal ranges. It follows that $M$
is not finite.
\end{proof}

The dichotomy in Corollary \ref{nov25} can also be formulated in terms
of the exit measure $m_t$. Indeed, $t$ is slim if and only if $m_t$ is
an atomic measure on $P(G)$. The constructions in \cite{Th5} show that
both possibilities occur in abundance.

It follows from Lemme 2.3.3 in \cite{C} that in the setting of Lemma
\ref{nov24} we have the identity
$$
\Gamma( \pi_{\psi_t}(C^*(G))'') = \Gamma(\mathcal R(t)) .
$$
Therefore Lemma \ref{nov24} also 
applies to determine
the factor type of the extremal $\beta$-KMS weights coming from an exit
for more general actions. The following gives an example of this.

\begin{example}\label{dec5b} 

Let $G$ be the following graph.
\begin{equation}\label{G5}
\begin{xymatrix}{
t_1 \ar@/^/[r]\ar@/_/[r] & t_2 \ar@/^/[r]\ar@/_/[r] \ar[d] & t_3 \ar[d]
\ar@/^/[r]\ar@/_/[r] & t_4 \ar[d]\ar@/^/[r]\ar@/_/[r] & t_5 \ar[d]
\ar@/^/[r]\ar@/_/[r] & \hdots \\
d_1 \ar[u]  &  d_2  \ar[l] & d_3  \ar[l] & d_4 
\ar[l] & d_5  \ar[l] & \hdots \ar[l] 
}
 \end{xymatrix}
\end{equation}
For each $n\in
  \mathbb N$ there are two edges $e^n_1$ and $e^n_2$ such that
  $s(e^n_i) = t_n, r(e^n_i) = t_{n+1}$, $i =1,2$. For each $n$ we set $F(e^n_i) =
  a_i$, where $a_1,a_2$ are positive real numbers with $a_1 > a_2$, and we set $F(e) =
  0$ for all other edges $e$. Let $\beta \in \mathbb R$ and consider the
  matrix $A(\beta)$ from (\ref{matrix}). Recall, \cite{Th5}, that a harmonic vector for $A(\beta)$ is a vector
$\psi = (\psi_v)_{v \in V}$ such that $\psi_v > 0$ and
$$
\psi_v = \sum_{w \in V} A(\beta)_{vw} \psi_w 
$$
for all $v \in V$.
Set 
$$
x_{\beta} = e^{-\beta a_1} + e^{-\beta a_2} .
$$
Consider a vector $\psi = (\psi_v)_{v \in V}$ with $\psi_{t_1} = 1$.
Then $\psi$ is a harmonic vector for $A(\beta)$ if and only if $\psi_v
> 0$ for all $v\in V$, and
\begin{enumerate}
\item[a)] $\psi_{d_i} = 1$ for all $i$,
\item[b)] $\psi_{t_2}= x_{\beta}^{-1}$, and
\item[c)] $\psi_{t_{k+1}} = x_{\beta}^{-1}(\psi_{t_k} -1)$, $k \geq 2$. 
\end{enumerate}
It follows that a harmonic vector exists if and only if $\beta \geq
\beta_0$, where $\beta_0$ is the real number with the property that
$$
x_{\beta_0} = e^{-\beta_0 a_1} + e^{-\beta_0 a_2} = \frac{1}{2} .
$$
By using the Criterion 2 in Corollary 4 on page 372 in \cite{V} we see
that $A(\beta)$ is $1$-transient for $\beta > \beta_0$ and
$1$-recurrent for $\beta = \beta_0$. 
%\bigskip
%Check: If $\beta > \beta_0$, take the harmonic vector $\psi'$ for
%$A(\beta')$ with $\psi'_{t_1} =1$, where $\beta' \in ]\beta,\beta_0]$. Then
%$$
%\sum_{w\in V} A(\beta)_{vw} \psi'_w \leq \sum_{w\in V} A(\beta')_{vw}
%\psi'_w = \psi'_v .
%$$
%By noting that these vectors values at $t_2$, it follows from the
%criterion of Vere-Jones that $A(\beta)$ is transient. On the other
%hand, $\varphi = (\varphi_v)$ is a postive vector with $\varphi_{t_1}
%= 1$ and $\sum_{w \in V} A(\beta_0)_{vw} \varphi_w \leq \varphi_v$ for
%all $v \in V$, it follows that $\varphi_{t_2} \leq 2$, $\varphi_{d_i}
%\geq 1$ for all $i$ and $\varphi_{t_{k+1}} \leq 2\left(\varphi_{t_k} -
%  \varphi_{d_{k+1}}\right) \leq  2\left(\varphi_{t_k} -
%  1\right)$ for all $k$. Since
%$$
%2 -\varphi_{t_{k+1}} \geq 2- 2(\varphi_{t_k} -1) = 2(2-
%\varphi_{t_k}),
%$$
%we see that if $\varphi_{t_k} < 2$ for some $k$, $\lim_{n \to \infty}
%\varphi_{t_n} = - \infty$, which is impossible. Hence $\varphi$ must
%be the harmonic vector, and Vere-Jones implies that $A(\beta_0)$ is $1$-recurrent.
%\bigskip
It follows then from Theorem 5.6 in \cite{Th5} that there is a unique
$\beta$-KMS weight $\psi_{\beta}$ for $\alpha^F$, up to multiplication
by scalars, for
all $\beta \geq \beta_0$; for $\beta > \beta_0$ it is given by the exit
measure coming from the unique exit in $G$ and for $\beta= \beta_0$ it is
given by a conservative measure on $P(G)$. By Lemma \ref{nov24} the
factor type of $\psi_{\beta}$ is $III_{\lambda}$ where $\lambda =
e^{-\beta(a_1  -a_2)}$ when $\beta >
\beta_0$. $\pi_{\psi_{\beta}}(C^*_r(G))''$ is a Powers factor in this
case. Concerning
the factor type of $\psi_{\beta_0}$ we note that the condition on $F$
in Theorem \ref{nov9} is satisfied and conclude that 
$$
\Gamma( \pi_{\psi_{\beta_0}}(C^*(G))'') = \overline{ \beta_0
  a_1\mathbb Z + \beta_0 a_2\mathbb Z} .
$$
Thus $\psi_{\beta_0}$ is of type $III_1$ if $\frac{a_1}{a_2}$ is
irrational. In other cases its type can be determined from 8.15.11 of \cite{Pe}.

%Jumping ahead a bit we note that it follows from Theorem \ref{nov14X} below
%that there is exactly one ground state for the restriction of $\alpha^F$ to
%the corner $1_{t_1}C^*(G)1_{t_1}$.

\end{example}

\section{Ground states}\label{ground}

In this section the graph $G$ can be an arbitrary countable oriented
graph and $F : E \to \mathbb R$ an arbitrary function. We consider the generalized gauge action $\alpha^F$ restricted to a
corner $1_vC^*(G)1_v$, where $v$ is a fixed but arbitrary vertex in the
graph. Recall that a state $\omega$ on $1_vC^*(G)1_v$ is a \emph{ground state}
for $\alpha^F$ restricted to $1_vC^*(G)1_v$ when
$$
-i \omega(a^* \delta(a)) \geq 0
$$
for all elements $a$ in the domain of $\delta$, the infinitesimal generator of (the
restriction of)
$\alpha^F$, cf. \cite{BR}. To describe the ground states observe first that the fix
point algebra of $\alpha^F$ is the $C^*$-algebra of the open sub-groupoid
$$
\mathcal F = \left\{ (\mu x, |\mu| - |\mu'|, \mu'x)  : \ x \in
  \Omega_G, \ F(\mu) = F(\mu')  \right\} 
$$  
of $\mathcal G$.
The conditional expectation $Q :C^*(G) \to C^*_r(\mathcal F)$ extending the
restriction map $C_c(\mathcal G) \to C_c(\mathcal F)$ can be
described as a limit:
\begin{equation}\label{intX}
Q(a) = \lim_{R \to \infty} \frac{1}{R} \int_{0}^{R} \alpha^F_t(a) \ dt ,
\end{equation}
cf. the proof of Theorem 2.2 in \cite{Th3}.

When $x \in \Omega_G, \ z \in P_f(G)$, write $z \subseteq x$ when $x|_{[1,|z|]} = z$ or $|z| =0$ and $z = s(x)$. An element $x \in \Omega_G$ is an \emph{$F$-geodesic} when the following holds:
$$
z,z' \in P_f(G), \ z \subseteq x, \ s(z') = s(z), \ r(z') = r(z) \ \Rightarrow \ F\left(z'\right) \geq  F\left(z\right) .
$$ 
We denote the set of $F$-geodesics in $\Omega_G$ by
$\text{Geo}(F,G)$. We leave the proof of the following observation to the reader.

\begin{lemma}\label{nov11}  $\text{Geo}(F,G) \cap Z(v)$ is closed in
  $Z(v)$, and $\mathcal F$-invariant in the sense that
$$
(x,k,y) \in \mathcal F, x \in \text{Geo}(F,G)\cap Z(v), \ y \in Z(v) \
\Rightarrow \ y \in \text{Geo}(F,G) .
$$
\end{lemma} 
%\begin{proof} 

%Let $\{x^n\}$ be a convergent sequence of geodesics in
%  $\Omega_G \cap Z(v)$; $x =\lim_{n \to \infty} x^n$. Let $y \in \Omega_G$ be an
%  element such that $s(y_1)= s(x_1) = v$ and $r(y_j) = r(x_i)$. There is an $n \in
%  \mathbb N$ such that $x^n_k = x_k$ for all $k \leq i$ and it
%  follows then that $j \geq i$ since $x^n$ is a geodesic. To see that
%  $\text{Geo(G)} \cap Z(v)$ is $\mathcal F$-invariant, assume that $x
%  \in Z(v)$ is a
%  geodesic and that $y \in Z(v)$ is an element such that $(x,0,y) \in
%  \mathcal F$. There is then an $i_0 \in
%  \mathbb N$ such that $r(x_{i_0}) = r(y_{i_0})$ and $x_k =y_k$ for all $k
%  > i_0$. To see that $y$ is a geodesic assume $z \in \Omega_G$ such
%  that $s(z) = v$ and $r(z_j) = r(y_i)$. We can then define $z' \in
%  Z(v)$ such that $z'_k = z_k, \ k \leq j$ and $z'_k = y_{i+k-j}$ when
%  $k > j$. There is then a $k \geq i_0 + j-i$ such that $r(z'_k) =
%  r(y_{i+k-j}) = r(x_{i+k-j})$ and it follows that $k \geq i+k-j$
%  because $x$ is a geodesic. Hence $j \geq i$ and $y$ is a geodesic.    

%\end{proof}

Set $\text{Geo}_v(F,G) = \text{Geo}(F,G) \cap Z(v)$. It follows from Lemma \ref{nov11} that the reduction $\mathcal
F|_{\text{Geo}_v(F,G)}$ of $\mathcal F$ to $\text{Geo}_v(F,G)$ is an \'etale groupoid in
itself and that there is a surjective $*$-homomorphism
$$
R : 1_vC^*_r(\mathcal F)1_v \to C^*_r\left(\mathcal
F|_{\text{Geo}_v(F,G)}\right)
$$
extending the map $1_v C_c(\mathcal F)1_v \to C_c(\mathcal
F|_{\text{Geo}_v(F,G)})$ given by restriction.

\begin{lemma}\label{nov12X} $-iR\circ Q(a^*\delta(a)) \geq 0$ for all
  $a$ in the domain of $\delta$.
\end{lemma}
\begin{proof} Since $1_vC_c(\mathcal G)1_v$ is a core
for $\delta$, cf. Corollary 3.1.7 in \cite{BR}, it suffices to show that
\begin{equation*}\label{nov13}
-i R \circ Q(f^*\delta(f)) \ \geq \ 0
\end{equation*}
when $f \in 1_vC_c(\mathcal G)1_v$. There is a finite collection of
paths $\mu_i,\nu_i \in P_f(G), i = 1,2, \cdots, m$, such that
$s(\mu_i) = s(\nu_i) = v$ for all $i$ and $f = \sum_{k
  =1}^m  f_k$, where
$$
\supp f_k \subseteq \left\{ (\mu_k x, |\mu_k| -|\nu_k|, \nu_k x) : \ x
  \in \Omega_G \right\} .
$$ 
Set
$$
A = \left\{ F(\mu_k) - F(\nu_k) : \ k = 1,2,\cdots, m\right\} .
$$
Then $f = \sum_{a \in A} h_a$, where $h_a \in 1_vC_c(\mathcal G)1_v$
has support in
$$
\bigcup_{ F(\mu_k) - F(\nu_k) =a} \  \left\{ (\mu_k x, |\mu_k| -|\nu_k|, \nu_k x) : \ x
  \in \Omega_G \right\} .
$$
In particular,
$\alpha^F_t(h_a) = e^{i at} h_a$
for all $t \in \mathbb R$.
It follows that
$$
-i Q(f^*\delta(f)) = \sum_{a,b \in A} bQ(h_a^*h_b) 
= \sum_{a \in A} a h_a^*h_a . 
$$
Note that 
$$
(x,k,y) \in \mathcal G, \ s(x) = v, \ y \in
\text{Geo}_v(F,G) \ \Rightarrow \ h_a(x,k,y) = 0
$$ 
when $a < 0$. It follows
that $R(h_a^*h_a) = 0$ when $a < 0$ and hence
$$
R( \sum_{a \in A} a h_a^*h_a) = R(\sum_{a \geq 0} ah_a^*h_a) \ \geq \ 0.
$$  
\end{proof}

\begin{thm}\label{nov14X} The map $\omega \mapsto \omega \circ R
  \circ Q$ is an affine homeomorphism from the state space of
  $C^*_r\left(\mathcal F|_{\text{Geo}_v(F,G)}\right)$ onto the set of
  ground states for the restriction of $\alpha^F$ to the corner $1_vC^*(G)1_v$.
\end{thm}
\begin{proof} Let $\tau$ be a ground state for the gauge action on
  $1_vC^*(G)1_v$. Since $\tau$ is $\alpha^F$-invariant by Proposition
  5.3.19 in \cite{BR} it follows that $\tau = \tau \circ Q$. Thanks to
  Lemma \ref{nov12X} it
  suffices therefore to show that the restriction of $\tau$ to
  $1_vC^*_r(\mathcal F)1_v$ factorises through $R$, i.e. that $\tau$
  annihilates the kernel of $R$, which is the reduction of $\mathcal F$ to
  $Z(v) \backslash \text{Geo}_v(F,G)$. (This fact follows from Remark
  4.10 in \cite{Re2} provided we know that $\mathcal
F|_{\text{Geo}_v(F,G)}$ is amenable. That this is true follows from
Theorem 4.2 in \cite{Pa} by use of results from \cite{A-DR}.) There is an approximate unit for
  $C^*_r\left(\mathcal F|_{Z(v) \backslash \text{Geo}_v(F,G)}\right)$ in $C_c(Z(v)
  \backslash \text{Geo}_v(F,G))$ and it suffices therefore to show that
  $\tau(f) = 0$ when $f  \in C_c(Z(v)
  \backslash \text{Geo}_v(F,G))$. Let $x \in \supp f$. Since $x \notin
  \text{Geo}_v(F,G)$ there is an element $h \in C_c(\mathcal G)$,
  supported in a bi-section of $\mathcal G$, such
  that $\alpha^F_t(h) = e^{iat} h$ for all $t \in \mathbb R$, where $a <
  0$, and such that $h^*h \in C_c(Z(v) \backslash \text{Geo}_v(F,G) )$ is constant $1$ in a
  neighborhood of $x$. Now
$$
0 \leq -i \tau(h^*\delta(h)) = a\tau(h^*h),
$$
because $\tau$ is a ground state, and it follows that $\tau(h^*h)=0$
since $a < 0$. As $x \in \supp f$ was
arbitrary, a standard partition of unity argument shows that $f$ is a
finite sum, $f = \sum_i g_if$, where $g_i \in  C_c(Z(v)
  \backslash \text{Geo}_v(F,G))$ are non-negative functions in the
  kernel of $\tau$. Hence $\tau(f) = 0$, as required.    
\end{proof}

To formulate what we get from Theorem \ref{nov14X} concerning the
ground states of the gauge action, we say that an element $x \in \Omega_G$ is a \emph{geodesic} when $x \in
V_{\infty}$ or the following holds:
$$
y \in \Omega_G, \ s(y_1) = s(x_1), \ r(y_j) = r(x_i) \ \Rightarrow \ j \geq i .
$$ 
We denote the set of geodesics in $\Omega_G$ by $\text{Geo}(G)$, and
set $\text{Geo}_v(G) = \text{Geo}(G) \cap Z(v)$. Set
$$
\mathcal R = \left\{ (\mu x, 0, \mu'x) \in \mathcal G : \ x \in
  \Omega_G \right\} ,
$$
which is the open sub-groupoid of $\mathcal G$ supporting the fixed
point algebra of the gauge action.

\begin{cor}\label{nov14} The set of ground states of the gauge action on
  $1_vC^*(G)1_v$ is affinely homeomorphic to the state space of
  $C^*_r\left(\mathcal R|_{\text{Geo}_v(G)}\right)$.
\end{cor}

\begin{remark}\label{nov33} Following Connes and Marcolli, \cite{CM},
  several authors have singled out the so-called $KMS_{\infty}$-states
  among the ground states for various one-parameter groups. A ground
  state is a $KMS_{\infty}$-state when it is the limit in the weak* topology
  of a sequence of $\beta_n$-KMS states, where $\lim_{n \to \infty}
  \beta_n = \infty$. By combining the description of the ground states
  in Theorem \ref{nov14X} with the observation that KMS states for
  $\alpha^F$ on the corner $1_vC^*(G)1_v$ are trace states on
  $C^*_r(\mathcal F)$, it follows that a ground state can only be a
  $KMS_{\infty}$-state if it is a trace. Presently it is not clear if this is the only condition in general.     
\end{remark}

If $G$ is a strongly connected graph with infinitely many vertexes,
the set $\text{Geo}_v(G)$ is not empty for any vertex $v$. Indeed, if
there is an infinite emitter $v_0$ in $V$, a path in $G$ from $v$ to
$v_0$ of minimal length is an element in $\text{Geo}_v(G) \cap Q(G)$,
and if there are no infinite emitters in $G$, the proof of Lemma 7.5 in
\cite{Th5} will give us an element of $\text{Geo}_v(G) \cap P(G)$.

\begin{example}\label{nov31} The following graph $G$ was considered in
  Example 5.1 in \cite{Th3}.

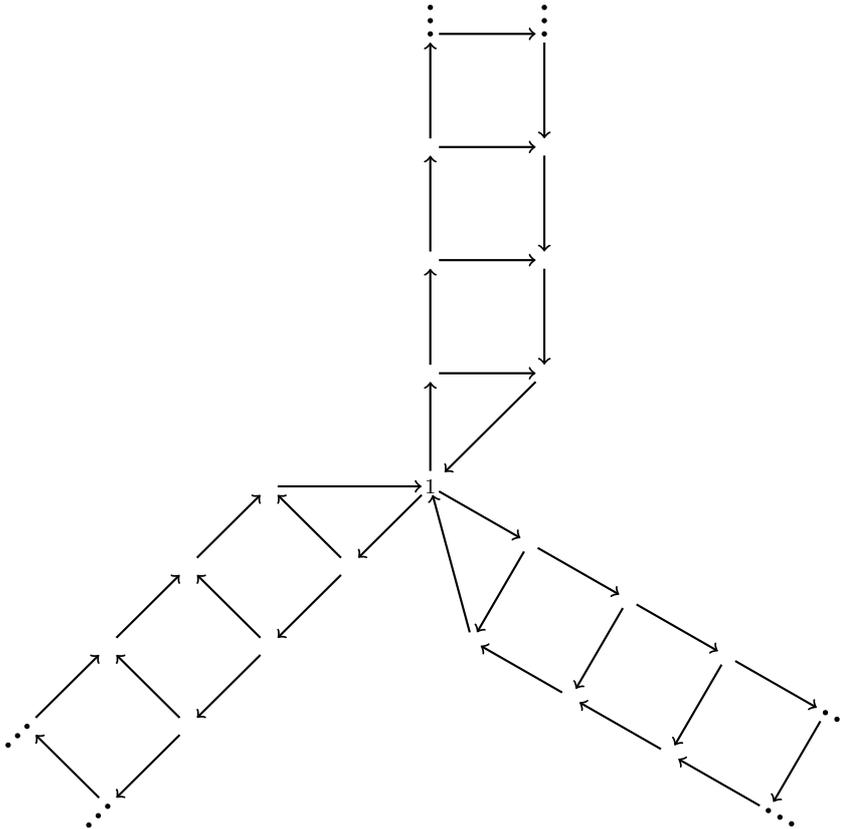
\bigskip

\begin{tiny}\label{dec5a}

\begin{tikzpicture}[->,thick,x=15mm,y=15mm]
  \begin{scope}
    \node (O) at (0,0) {$1$};
    \node (a1) at (0,1) {};
    \node (a2) at (0,2) {};
    \node (a3) at (0,3) {};
    \node (a4) at (0,4) {};
    \node (am1) at (1,1) {};
    \node (am2) at (1,2) {};
    \node (am3) at (1,3) {};
    \node (am4) at (1,4) {};

    \draw  (O)--(a1);
    \draw (a1)--(a2);
    \draw (a2)--(a3);
    \draw (a3)--(a4);
    \draw (a1)--(am1);
    \draw (a2)--(am2);
    \draw (a3)--(am3);
    \draw (a4)--(am4);
    \draw (am4)--(am3);
    \draw (am3)--(am2);
    \draw (am2)--(am1);
    \draw (am1)--(O);
    \fill 
    (a4) circle (1pt)
    (a4)++(0,5pt) circle (1pt)
    (a4)++(0,10pt) circle (1pt)
    (am4) circle (1pt)
    (am4)++(0,5pt) circle (1pt)
    (am4)++(0,10pt) circle (1pt)
    ;

%    \node at (O) [right=2mm] {$1$};
  \end{scope}
  \begin{scope}[rotate=-120]
    \node (O) at (0,0) {};
    \node (c1) at (0,1) {};
    \node (c2) at (0,2) {};
    \node (c3) at (0,3) {};
    \node (c4) at (0,4) {};
    \node (cm1) at (1,1) {};
    \node (cm2) at (1,2) {};
    \node (cm3) at (1,3) {};
    \node (cm4) at (1,4) {};
    
    \draw  (O)--(c1);
    \draw (c1)--(c2);
    \draw (c2)--(c3);
    \draw (c3)--(c4);
    \draw (c1)--(cm1);
    \draw (c2)--(cm2);
    \draw (c3)--(cm3);
    \draw (c4)--(cm4);
    \draw (cm4)--(cm3);
    \draw (cm3)--(cm2);
    \draw (cm2)--(cm1);
    \draw (cm1)--(O);
    \fill 
    (c4) circle (1pt)
    (c4)++(0,5pt) circle (1pt)
    (c4)++(0,10pt) circle (1pt)
    (cm4) circle (1pt)
    (cm4)++(0,5pt) circle (1pt)
    (cm4)++(0,10pt) circle (1pt)
    ;
  \end{scope}

  \begin{scope}[rotate=135]
    \node (O) at (0,0) {};
    \node (b1) at (0,1) {};
    \node (b2) at (0,2) {};
    \node (b3) at (0,3) {};
    \node (b4) at (0,4) {};
    \node (bm1) at (1,1) {};
    \node (bm2) at (1,2) {};
    \node (bm3) at (1,3) {};
    \node (bm4) at (1,4) {};
    
    \draw  (O)--(b1);
    \draw (b1)--(b2);
    \draw (b2)--(b3);
    \draw (b3)--(b4);
    \draw (b1)--(bm1);
    \draw (b2)--(bm2);
    \draw (b3)--(bm3);
    \draw (b4)--(bm4);
    \draw (bm4)--(bm3);
    \draw (bm3)--(bm2);
    \draw (bm2)--(bm1);
    \draw (bm1)--(O);
    \fill 
    (b4) circle (1pt)
    (b4)++(0,5pt) circle (1pt)
    (b4)++(0,10pt) circle (1pt)
    (bm4) circle (1pt)
    (bm4)++(0,5pt) circle (1pt)
    (bm4)++(0,10pt) circle (1pt)
    ;
  \end{scope}
\end{tikzpicture}
\end{tiny}

\bigskip

Let $\alpha$ be the real root of the polynomial $x^3 -x-3$. As shown
in \cite{Th3} there are three extremal rays of $\beta$-KMS weights for
the gauge action on $C^*(G)$ when $\beta > \log \alpha$ and a unique
ray of
$\log \alpha$-KMS weights. In fact, the last ray consists of finite
$\log \alpha$-KMS weights that may be normalised to a $\log
\alpha$-KMS state. For $\beta > \log \alpha$ the weights are not
finite. The factor type of the KMS weights were not determined in
\cite{Th3}, but it follows from Theorem \ref{nov9} that the factor type of
the $\log \alpha$-KMS state is $III_{\alpha^{-1}}$ while the other extremal $\beta$-KMS weights are all of
type $I_{\infty}$ by Theorem \ref{nov25}.

Let $v$ be the central vertex in $G$; the one labelled $1$. Then $ \text{Geo}_v(G)
$ consists of three easily identified paths in $G$ and
$$
C^*_r\left(\mathcal R|_{\text{Geo}_v(G)}\right) \ \simeq \ \mathbb C
\oplus \mathbb C \oplus \mathbb C.
$$
By Corollary \ref{nov14} the set of ground states for the gauge action
on $1_vC^*(G)1_v$ is a triangle. %In this case they are all $KMS_{\infty}$-states.

\end{example}

\begin{example}\label{nov34} Consider the gauge action on $C^*(G)$
  when $G$ is the graph (\ref{G5}).  A vector $\psi_v$ indexed by the
  vertexes in $G$ with $\psi_{t_1} =1$ is harmonic for $A(\beta)$ when
  $\psi_v > 0$ for all $v\in V$ and
\begin{enumerate}
\item[1)] $\psi_{d_n} = e^{-n\beta}, \ n =1,2,3,\cdots$,
\item[2)] $\psi_{t_2} = \frac{e^{\beta}}{2}$,
\item[3)] $\psi_{t_{k+1}} = \frac{1}{2} \left(e^{\beta} \psi_{t_k} -
    e^{-k\beta}\right), \ k \geq 2$.   
\end{enumerate}
It follows from this (and \cite{Th5}) that there is a $\beta$-KMS
weight $\psi_{\beta}$ for the gauge action on $C^*(G)$ if and only
$\beta \geq \frac{1}{2} \log
\alpha$, where $\alpha \sim 1.61$ is the positive root of the polynomial $x^2 -2x
-2$, and that it is unique up to multiplication by
scalars. In addition, it follows from \cite{Th5} that only the $\frac{\log \alpha}{2}$-KMS weight can arise from a
conservative measure on $P(G)$, and by using the criterion from
\cite{V} one checks that it does. It follows then from Theorem \ref{nov25} that
$\psi_{\beta}$ is of factor type $II_{\infty}$ for all $\beta >
\frac{\log \alpha}{2}$ while the factor type of $\psi_{\beta}$ is
$III_{\alpha^{-1}}$ when $\beta = \frac{\log \alpha}{2}$ by Theorem \ref{nov9}.

Concerning the ground states it follows from Corollary \ref{nov14}
that the ground states for the gauge action on the corner
$1_{t_1}C^*(G)1_{t_1}$ are parametrised by the state space of the UHF algebra of type
$2^{\infty}$, also known as the CAR algebra. For comparison, let
$\alpha^F$ be the generalized gauge action on $C^*(G)$ considered in
Example \ref{dec5b}. The restriction of $\alpha^F$ to the same corner $1_{t_1}C^*(G)1_{t_1}$
has only one ground state by Theorem \ref{nov14X}.

%Check: Set $z_k = e^{k\beta} t_k$. Then 3) becomes
%$$
%z_{k+1} = \frac{1}{2} \left( a^2 z_k - a\right) = f(z_k)
%$$
%where $a = e^{\beta}$. If $a^2 \leq 2$ we see that $z_{k+1} \leq z_k -
%\frac{a}{2}$ which implies that $z_k < 0$ for $k$ large - which is
%impossible. Therefore $a^2 > 2$. Set
%$$
%r =\frac{\frac{1}{2} a}{\frac{1}{2} a^2 -1}
%$$
%and note that $f(r) = r$. If $z_2 < r$ we find that
%$$
%z_3 - r = f(z_2) - f(r) = \frac{a^2}{2} (z_2 - r), 
%$$
%generally $z_k = \left(\frac{a^2}{2}\right)^{k-2} (z_2 -r)$, which
%again leads to a contradiction. The other possibility, $z_2 \geq r$
%leads to solution. So there is a (unique) solution if and only if
%$$
%z_2 = e^{2\beta} t_2 = \frac{e^{3\beta}}{2} \geq \frac{\frac{1}{2}
%  a}{\frac{1}{2} a^2 -1} .
%$$
%Or $a^2 \geq \frac{1}{\frac{1}{2}a^2 -1} \ \Leftrightarrow \
%\frac{1}{2}a^4 - a^2 \geq 1$. This occurs if and only $a^2 \geq
%\alpha$ where $\alpha$ is the positive root of $x^2 -2x-2$. 

%When $\beta = \frac{\log \alpha}{2}$ we find that $z_k = z_2$ for all
%$k \geq 2$, i.e.
%$$
%t_2 = \frac{\sqrt{\alpha}}{2} , \ z_2 = t_2e^{2\beta} = \frac{e^{3\beta}}{2}
%$$
%and
%$$
%t_k = e^{-k\beta} z_k = e^{-k \beta} \frac{e^{3\beta}}{2} =
%\frac{e^{(3-k)\beta}}{2} = \frac{\alpha^{\frac{3-k}{2}}}{2} 
%$$
%for all $k \geq 3$.

\end{example}

\begin{example}\label{dec5} By using amalgamations of graphs in much
  the same way as in \cite{Th5}, it is possible to build examples
  with other kinds of variation of the factor types. For example we
  can construct cases where there are extremal KMS weights of
  different factor types for the same value of $\beta$. For this consider the disjoint union of the graph (\ref{G5}) from Example \ref{dec5b} and the graph (\ref{dec5a})
  from Example \ref{nov31}, and identify the vertex $t_1$ from
  (\ref{G5}) with the vertex $1$ from (\ref{dec5a}) to get a strongly
  connected graph $H$. Consider the
  generalised gauge action on $C^*(H)$ obtained by taking $F(e)$ to be
   the value it
  has in Example \ref{dec5b} when $e$ comes from (\ref{G5}) and to be
  $1$ when $e$ comes from (\ref{dec5a}). For the corresponding action
  $\alpha^F$ on $C^*(H)$ there are
  $\beta$-KMS weights if and only if $\beta \geq \beta_0$ where
  $\beta_0$ is the positive real number for which
$$
\frac{3 e^{-\beta_0}}{e^{2\beta_0} -1} + \frac{e^{-\beta_0 a_1} + e^{-\beta_0
    a_2}}{1 - e^{-\beta_0 a_1} -e^{-\beta_0 a_2}} = 1.
$$  
When $\beta > \beta_0$ there are four extremal rays of $\beta$-KMS weights,
three of which are of type $I_{\infty}$ and one of type
$III_{\lambda}$ where $\lambda = e^{-\beta (a_1-a_2)}$. There is only
one ray of $\beta_0$-KMS weights, and by using the results of
Vere-Jones as in Example \ref{dec5b} we can conclude from \cite{Th5}
that the corresponding measure on $P(H)$ is conservative. We can then use Theorem \ref{nov9}
to see that the $\Gamma$-invariant of the corresponding factor is 
$$
\overline{\beta_0 a_1 \mathbb Z + \beta_0 a_2 \mathbb Z + \beta_0
  \mathbb Z} .
$$
Its type depends therefore on the two parameters $a_1,a_2$, and can be determined from 8.15.11 of \cite{Pe}.
\end{example}


\begin{thebibliography}{WWWWW} %antallet af W'er er bredeste ind



%\bibitem[An]{An} C. Anantharaman-Delaroche, {\em Purely infinite $C^*$-algebras arising from dynamical systems}, Bull. Soc. Math. France {\bf 125} (1997), 199--225.

\bibitem[A-DR]{A-DR} C. Anantharaman-Delaroche and J. Renault, {\em
    Amenable groupoids}, Monographies de L'Enseignement Mathï¿½matique,
  vol. 36, L'Enseignement Mathï¿½matique, Geneva, 2000. 




%\bibitem[BPRS]{BPRS} T. Bates, D. Pask, I. Raeburn and W. Szymanski,
%  {\em The $C^*$-algebra of Row-Finite Graphs}, New York Jour. of
%  Math. {\bf 6} (2000), 307-324.

\bibitem[AW]{AW} H. Araki and E.J. Woods, {\em A classification of
    factors}, Publ. Res. Inst. Math. Sci. Ser. A {\bf 4} (1968),
  51-130.



\bibitem[BR]{BR} O. Bratteli and D.W. Robinson, {\em Operator Algebras
    and Quantum Statistical Mechanics I + II}, Texts and Monographs in
  Physics, Springer Verlag, New York, Heidelberg, Berlin, 1979 and 1981.


\bibitem[BEK]{BEK} O. Bratteli, G. Elliott and A. Kishimoto, {\em The
    temperature state space of a dynamical system I}, J. Yokohama
  Univ. {\bf 28} (1980), 125-167. 


\bibitem[CL]{CL} T.M. Carlsen and N. Larsen, {\em Partial actions and
    KMS states on relative graph $C^*$-algebras}, arXiv:1311.0912. 

%\bibitem[Co]{Co} D.L. Cohn, {\em Measure Theory}, Birkh\"auser,
%  Boston, Basel, Berlin, 1997.


%\bibitem[Ca]{Ca} T.M. Carlsen, {\em Talk at the Operator Algebra and
%    Dynamics Conference}, Faroe Islands, May 2012.

\bibitem[Co]{Co} F. Combes, {\em Poids associ\'e \`a une alg\`ebre
    hilbertienne \`a gauche}, Compos. Math. {\em 23} (1971), 49-77.


\bibitem[CT]{CT} J. Christensen and K. Thomsen, {\em Finite digraphs and
    KMS states}, arXiv, May 2015.

  
\bibitem[C]{C} A. Connes, {\em Une classification des facteurs de type III}, Ann. Sci. Ecole Norm. Sup. {\bf 6} (1973), 133-252.



\bibitem[CM]{CM} A. Connes and M. Marcolli, {\em Quantum statistical
    mechanics of $\mathbb Q$-lattices.} In Frontiers in Number Theory,
  Physics, and Geometry I, Springer-Verlag, 2006, pp 269-349.
% \bibitem[C2]{C2} A. Connes, {\em Classification of injective factors}, Ann. Math. {\bf 104} (1976), 73-115. 
 
 
%\bibitem[Cy]{Cy} V.T. Cyr, {\em Transient Markov Shifts},
%  Ph.D. thesis, Pennsylvania State University, August 2010.

%\bibitem[De]{De} V. Deaconu, {\em Groupoids associated with endomorphisms}, Trans. Amer. Math. Soc. {\bf 347} (1995), 1779-1786.



%bibitem[DS]{DS} V. Deaconu and F. Shultz, {\em $C^*$-algebras
%   associated with interval maps}, Trans. Amer. Math. Soce. {\bf 359}
% (2007), 1889-1924.


%\bibitem[DU]{DU} M. Denker and M. Urbanski, {\em On the Existence of
%    Conformal Measures}, Trans. Amer. Math. Soc. {\bf 328} (1991), 563-587.


\bibitem[EFW]{EFW} M. Enomoto, M. Fujii and Y. Watatani, {\em KMS
    states for gauge action on $O_A$}, Math. Japon. {\bf 29} (1984), 607-619.




\bibitem[EL]{EL} R. Exel and M. Laca, {\em Partial dynamical systems
    and the KMS condition}, Comm. Math. Phys. {\bf 232} (2003), 223-277.


%\bibitem[GV]{GV} B.M. Gurevich and S.V. Savchenko, {Thermodynamic
%    formalism for countable symbolic Markov chains}, Russian
%  Math. Surveys {\bf 53} (1998), 245-344.

%\bibitem[K]{K} A. Kumjian, {\em Notes on graph $C^*$-algebras},
%  Operator Algebras and Operator Theory (Shanghai,1997),
%  Contemp. Math. {\bf 228}, 189-200.


%\bibitem[H]{H} T.E. Harris, {\em Transient Markov chains with stationary measures}, 
%Proc. Amer. Math. Soc. {\bf 8} (1957), 937-942. 

\bibitem[aHLRS]{aHLRS} A. an Huef, M. Laca, I. Raeburn and A. Sims,
  {\em KMS states on the $C^*$-algebras of reducible graphs},
   Ergodic Th. \& Dynam. Syst., to appear.



%\bibitem[aHLRS]{aHLRS} A. an Huef, M. Laca, I. Raeburn and A. Sims,
%  {\em KMS states on the $C^*$-algebras of a finite graph},
%  J. Math. Anal. Appl. {\bf 405} (2013), 388-399.

%\bibitem[I]{I} M. Izumi, {\em Subalgebras of infinite $C^*$-algebras with
%    finite Watatani index I}, Comm. Math. Phys. {\bf 155} (1993), 157-182.


\bibitem[Ka]{Ka} E. T.A. Kakariadis, {\em On Nica-Pimsner algebras of $C^*$-dynamical systems
  over $\mathbb Z^N_+$}, arXiv:1411.4992.


\bibitem[KR]{KR} R.V. Kadison and J.R. Ringrose, {\em Fundamentals of
    the Theory of Operator Algebras II}, Academic Press, London 1986.


%\bibitem[KW]{KW} T. Kajiwara and Y. Watatani, {\em KMS states on
%    finite-graph $C^*$-algebras}, arXiv:1007.4248.


\bibitem[K]{K} A. Kishimoto, {\em On uniqueness of KMS states for
    one-dimensional quantum lattice systems}, Comm. Math. Phys. {\bf
    47} (1976), 167-170.

%\bibitem[K]{K} B.P. Kitchens, {\em Symbolic Dynamics}, Springer
%  Verlag, Berlin, Heidelberg, 1998. 



%\bibitem[Kr]{Kr} U. Krengel, {\em Ergodic Theorems}, de Gruyter
%  Studies in Mathematics, Berlin 1985.


%\bibitem[KPRR]{KPRR} A. Kumjian, D. Pask, I. Raeburn and J. Renault, {\em Graphs, Groupoids, and Cuntz-Krieger algebras}, J. Func. Analysis {\bf 144} (1997), 505-541.   

%\bibitem[KR]{KR} A. Kumjian and J. Renault, {\em KMS states on
%    $C^*$-algebras associated to expansive maps},
%  Proc. Amer. Math. Soc. {\bf 134} (2006), 2067-2078.


%\bibitem[HKO]{HKO} G. Hara\'anczyk, D. Kwietniak and P. Oprocha, {\em
%    Topological structure and entropy of mixing graph maps'}, arXiv:1111.0566v1.




%bibitem[Ka]{Ka} T. Katsura, {\em On $C^*$-algebras associated with
%    $C^*$-correspondences}, J. Func. Analysis {\bf 217} (2004), 366-401.



%\bibitem[Ku]{Ku} J. Kustermans, {\em KMS weights on $C^*$-algebras},
%  arXiv: 9704008v1.


\bibitem[KV1]{KV1} J. Kustermans and S. Vaes, {\em Weight theory for
    $C^*$-algebraic quantum groups}, arXiv:990163.



\bibitem[KV2]{KV2} J. Kustermans and S. Vaes, {\em Locally compact
    quantum groups}, Ann. Scient. \'Ec. Norm. Sup. {\bf 33}, (2000), 837-934.

% \bibitem[LN]{LN} M. Laca and S. Neshveyev, {\em KMS states of
%     quasi-free dynamics on Pimsner algebras}, J. Func. Analysis {\bf
%     211} (2004), 457-482.

\bibitem[LR]{LR} M. Laca and I. Raeburn, {\em Phase transition on the
    Toeplitz algebra of the affine semigroup over the natural
    numbers}, Adv. Math. {\bf 225} (2010), 643-688.


%\bibitem[M]{M} H. Minc, {\em Nonnegative Matrices}, John Wiley \&
%Sons, New York, 1988.

%\bibitem[MRW]{MRW} P. Muhly, J. Renault and D. Williams, {\em Equivalence and isomorphism for groupoid $C^*$-algebras}, J. Oper. Th. {\bf 17} (1987), 3--22.

\bibitem[N]{N} S. Neshveyev, {\em KMS states on the $C^*$-algebras of
    non-principal groupoids}, J. Operator Theory {\bf 70} (2013), 513-530.


\bibitem[O]{O} R. Okayasu, {\em Type III factors arising from
    Cuntz-Krieger algebras}, Proc. Amer. Math. Soc. {\bf 131} (2003), 2145-2153. 

\bibitem[OP]{OP} D. Olesen and G.K. Pedersen, {\em Some
    $C^*$-dynamical systems with a single KMS state},
  Math. Scand. {\bf 42} (1978), 111-118.



%\bibitem[Pa]{Pa} K. R. Parthasarathy, {\em Introduction to Probability
%    and Measure}, Macmillan India Press, 1977.

\bibitem[Pa]{Pa} A.L.T. Paterson, {\em Graph inverse semigroups,
    groupoids and their $C^*$-algebras}, J. Operator Theory {\bf 48}
  (2002), 645-662.

\bibitem[Pe]{Pe} G. K. Pedersen, {\em $C^*$-algebras and their
    automorphism groups}, Academic Press, London, 1979.


\bibitem[PS]{PS} R.T. Powers and S. Sakai, {\em Existence of ground
    states and KMS states for approximately inner dynamics},
  Comm. Math. Phys. {\bf 39} (1975), 273-288.



%\bibitem[Pr]{Pr} W.E. Pruitt, {\em Eigenvalues of non-negative
%    matrices}, Ann. Math. Statist. {\bf 35} (1964), 1797-1800.


%\bibitem[QV]{QV} J. Quaegebeur and J. Verding, {\em A construction for
%    weights on $C^*$-algebras. Dual weights for $C^*$-crossed
%    products}, Inter. J. Math. {\bf 10} (1999), 129-157.


\bibitem[Re1]{Re1} J. Renault, {\em A Groupoid Approach to $C^*$-algebras},  LNM 793, Springer Verlag, Berlin, Heidelberg, New York, 1980.  

\bibitem[Re2]{Re2} J. Renault, {\em The ideal structure of groupoid
    crossed product $C^*$-algebras}, J. Oper.Theory {\bf 25} (1991),
  3-36.

%\bibitem[Re2]{Re2} J. Renault, {\em AF equivalence relations and their
%    cocycles}, Operator algebras and mathematical physics
%  (Constanta, 2001), 365-377, Theta, Bucharest, 2003.

%bibitem[RS]{RS} J. Rosenberg and C. Schochet, {\em The K\"unneth theorem and the universal coefficient theorem for Kasparov's generalized K-functor}, Duke J. Math. {\bf 55} (1987), 337-347.

%\bibitem[Ru]{Ru} S. Ruette, {\em On the Vere-Jones classification and
%    existence of maximal measure for countable topological Markov
%    chains}, Pac. J. Math. {\bf 209} (2003), 365-380.

%\bibitem[S]{S} O. Sarig, {\em Thermodynamic formalism for
%    countable Markov shifts}, Ergodic Th. \& Dynam. Syst. {\bf 19}
%  (1999), 1565-1593.



%\bibitem[Sa2]{Sa2} O. Sarig, {\em Lecture Notes on Thermodynamic
%    Formalism for Toplogical Markov Shifts}, Penn State, 2009.


%\bibitem[Sa]{Sa} S.A. Sawyer, {\em Martin boundary and random walks},
%  Harmonic functions on trees and buildings (New York, 1995), 17-44,
%  Contemp. Math. 206, Amer. Math. Soc., Providence, RI, 1997.
 
%\bibitem[S]{S} F. Shultz, {\em Dimension groups for interval maps II:
%    The transitive case}, Ergod. Th. \& Dynam. Sys. {\bf 27} (2007), 1287-1321.

%\bibitem[Sp]{Sp} J. Spielberg, {\em Groupoids and $C^*$-algebras for categories of paths}, ArXiv:1111.6924v3.

%\bibitem[SW]{SW} A. Sims and D. Williams, {\em Renault's equivalence theorem for reduced groupoid $C^*$-algebras}, arXiv:1002.3093v2
 

%\bibitem[ST]{ST} T. L. Schmidt and K. Thomsen, {\em Circle maps and $C^*$-algebras},

%\bibitem[Sz]{Sz} W. Szymanski, {\em Simplicity of Cuntz-Krieger
 %   algebras of infinite matrices}, Pac. J. Math. {\bf 122} (2001), 249-256.

\bibitem[Th1]{Th1} K. Thomsen, {\em Semi-\'etale groupoids and
  applications}, Annales de l'Institute Fourier {\bf 60} (2010), 759-800.


%\bibitem[Th2]{Th2} K. Thomsen, {\em On the $C^*$-algebra of a Locally
%    Injective Surjection and its KMS states}, Comm. Math. Phys. {\bf
%    302} (2011), 403-423.


%\bibitem[Th2]{Th2} K. Thomsen, {\em $KMS$-states and conformal
%   measures}, Comm. Math. Phys. {\bf 316} (2012), 615-640. DOI:
%   10.1007/s00220-012-1591-z 

\bibitem[Th2]{Th2} K. Thomsen, {\em Exact circle maps and KMS states},
  Israel J. Math. {\bf 205} (2015), 397-420.  DOI: 10.1007/s11856-014-1124-x


\bibitem[Th3]{Th3} K. Thomsen, {\em KMS weights on groupoid and graph
    $C^*$-algebras}, J. Func. Analysis {\bf 266} (2014), 2959-2988.





%\bibitem[Th2]{Th2} K. Thomsen, {\em On the positive eigenvalues and
%    eigenvectors of a non-negative matrix}, arXiv:1306.5116.


\bibitem[Th4]{Th4} K. Thomsen, {\em Dissipative conformal measures on
    locally compact spaces},  Ergodic Th. \& Dynam. Syst., to appear.


\bibitem[Th5]{Th5} K. Thomsen, {\em KMS weights on graph
    $C^*$-algebras}, arXiv:1409.3702

%\bibitem[To]{To} M. Tomforde, {\em The structure of graph $C^*$-algebras and their generalizations}, ????


\bibitem[V]{V} D. Vere-Jones, {\em Ergodic properties of
    non-negative matrices I}, Pacific J. Math. {\bf 22} (1967), 361-386.

%\bibitem[W]{W} P. Walters, {\em An Introduction to Ergodic Theory}, Springer Verlag, New York, Heidelberg, Berlin, 1982.

%\bibitem[Wo]{Wo} W. Woess, {\em Denumerable Markov Chains}, EMS
%  Textbooks in Mathematics, 2009.

%\bibitem[Z]{Z} J. Zacharias, {\em Quasi-free automorphisms on
%    Cuntz-Krieger-Pimsner algebras}, in $C^*$-algebras (M\"unster, 1999), 262-272.

\end{thebibliography}
\end{document}